\documentclass[12pt]{article}

\usepackage{amsmath,amssymb,amsthm}

\usepackage{eucal}
\usepackage[cmtip,all]{xy}

\usepackage[pdftex]{hyperref} %, hyperindex=true, backref=page

\hypersetup{
  colorlinks = true,
  linkcolor = black,
  urlcolor = blue,
  citecolor = black,
}

\usepackage{mathptmx}
\usepackage[scaled=.90]{helvet}
\usepackage{courier}

\newtheorem{theorem}{Theorem}[section]

\newtheorem{proposition}[theorem]{Proposition}

\theoremstyle{definition}

\newtheorem{definition}[theorem]{Definition}
\newtheorem{remark}[theorem]{Remark}
\newtheorem{example}[theorem]{Example}

\numberwithin{equation}{section}

\renewcommand{\phi}{\varphi}

\newcommand{\Map}{\operatorname{Map}}
\newcommand{\Pic}{\operatorname{Pic}}
\newcommand{\Sym}{\operatorname{Sym}}

\newcommand{\vol}{\operatorname{vol}}

\newcommand{\calA}{\mathcal{A}}
\newcommand{\calG}{\mathcal{G}}
\newcommand{\calO}{\mathcal{O}}
\newcommand{\calM}{\mathcal{M}}
\newcommand{\calV}{\mathcal{V}}

\newcommand{\B}{\operatorname{B}}
\newcommand{\SO}{\operatorname{SO}}
\newcommand{\BSO}{\operatorname{BSO}}
\newcommand{\Spinc}{\operatorname{Spin}^\mathrm{c}}
\newcommand{\U}{\operatorname{U}}
\newcommand{\Aut}{\operatorname{Aut}}
\newcommand{\BAut}{\operatorname{BAut}}
\newcommand{\Diff}{\operatorname{Diff}}
\newcommand{\BDiff}{\operatorname{BDiff}}

\newcommand{\upF}{\mathrm F}
\newcommand{\upH}{\mathrm H}
\newcommand{\rmS}{\mathrm S}

\newcommand{\upZ}{\mathrm Z}

\renewcommand{\c}{\mathrm c}
\renewcommand{\d}{\mathrm d}

\newcommand{\h}{\mathrm h}

\newcommand{\cci}{\mathrm i}

\newcommand{\CC}{\mathbb C}
\newcommand{\RR}{\mathbb R}
\newcommand{\GG}{\mathbb G}
\newcommand{\BGG}{\mathrm B\mathbb G}
\newcommand{\TT}{\mathbb T}
\newcommand{\BTT}{\mathrm B\mathbb T}
\newcommand{\ZZ}{\mathbb Z}

\newcommand{\p}{\partial^{\phantom*}}
\newcommand{\pstar}{\partial^*}
\newcommand{\pbar}{\overline\partial^{\phantom*}}
\newcommand{\pbarstar}{\overline\partial^*}

\author{Markus Szymik}

\title{The stable homotopy theory of \hbox{vortices on Riemann
    surfaces}}
    
\date{October 2013}

\begin{document}

\maketitle

\renewcommand{\abstractname}{}

\begin{abstract}
  \noindent
The purpose of these notes is to show that the methods introduced by Bauer and Furuta in order to refine the Seiberg-Witten invariants of smooth~$4$-dimensional manifolds can also be used to obtain stable homotopy classes from Riemann surfaces, using the vortex equations on the latter. 
\end{abstract}

\thispagestyle{empty}

%%%%%%%%%%%%%%%%%%%%%%%%%%%%%%%%%%%%%%%%%%%%%%%%%%%%%%%%%%%%%%%%%%

\section*{Introduction}

The purpose of these notes is to show that the methods introduced by Bauer and Furuta, see~\cite{Bauer,Bauer:survey,BauerFuruta}, in order to refine the Seiberg-Witten invariants of smooth~$4$-dimensional manifolds can also be used to obtain stable homotopy classes from~2-dimensional manifolds, using the vortex equations on the latter. 

So far, these notes contain barely more than some background material and the necessary analytic estimates to prove this. In Section~\ref{sec:dimension1}, we review some basic facts concerning abelian gauge theory in real dimension 2, and introduce some notation along the way. Gauge symmetries will be discussed in the following Section~\ref{sec:gauge_symmetries}. In Section~\ref{sec:equations}, the vortex equations are introduced. They form a coupled system of non-linear partial differential equations, and this will be packaged into the vortex map in Section~\ref{sec:map}. All analytic estimates are concentrated in the following Section~\ref{sec:estimates}. They imply that the Bauer-Furuta construction can be used to define stable homotopy invariants. This is explained in Section~\ref{sec:invariants}. In Section~\ref{sec:actions_and_families} we sketch how to extend this to obtain invariants for symmetries (group actions) and deformations (families and moduli) of Riemann surfaces. In the final Section~\ref{sec:g=0}, this is specialized to surfaces of genus~$0$. 

Implications and applications will be added as time permits. For example, a generalization of the equations studied here, the so called {\em symplectic vortex equations}, have been introduced in order to define gauged Gromov-Witten invariants for symplectic quotients of Hamiltonian group actions, see~\cite{CGMS}. The case studied in this paper is the standard scalar circle action on the complex line, where the Marsden-Weinstein quotient reduces to a point.

Some of the results in this paper date back from 2004, and the author would like to thank Stefan Bauer and Kim Fr\o yshov for discussions in these early days. Since then,~I have had the pleasure to discuss the material with many other people, and I would in particular like to thank Marek Izydorek, Oscar Randal-Williams, and Maciej Starostka for their interest. I am grateful to Oscar Garc\'\i a-Prada for making~\cite{GarciaPrada:Seiberg-Witten1} and~\cite{GarciaPrada:Seiberg-Witten2} available to me. The production of these notes from a long-dormant draft has been greatly supported by Nuno Rom\~{a}o, who invited me to the Hausdorff Trimester Program on Mathematical Physics, where I have had the pleasure to come into contact with a re-flourishing vortex community, and I thank him heartily for his hospitality. This research has also been supported by the Danish National Research Foundation through the Centre for Symmetry and Deformation (DNRF92).

%%%%%%%%%%%%%%%%%%%%%%%%%%%%%%%%%%%%%%%%%%%%%%%%%%%%%%%%%%%%%%%%%%

\section{Line bundles on Riemann surfaces}\label{sec:dimension1}

In this section we review the necessary background material about connections on line bundles over Riemann surfaces.

\subsection{Riemann surfaces} 

Let~$X$ be a closed oriented connected surface with a Riemann metric.  The
Hodge star operator~$\star$ on~$1$-forms induces a complex structure on~$X$.
This structure is automatically integrable, so that~$X$ is a complex curve. The
metric on~$X$ is automatically K\"ahler, and the K\"ahler
form~$\omega=\star1$ agrees with the volume form. In particular,
\[
\int_X\omega=\vol(X).
\]

If~$\phi$ is a section in some Hermitian vector bundle on~$X$, then we write~$|\phi|$ for the function on~$X$ which assigns to a point~$x$ the length~$|\phi(x)|$ of~$\phi(x)$, and we write
\[
\|\phi\|^2=\int_X|\phi|^2,
\]
omitting---as is often done---the notation for the volume form from the integrand.

\subsection{Unitary connections on Hermitian line bundles} 

Let~$L$ be a complex line bundle over a Riemann surface~$X$ as above, equipped with a Hermitian metric. Up to isomorphism, these structures are determined by the degree~$\deg(L)$ of~$L$. 

The curvature form~$\upF_A$ of a unitary connection~$A$ on~$L$ is an imaginary~$2$-form on~$X$. If~$B$ is a chosen base connection, and~$A=B+\cci\alpha$ for some real~$1$-form~$\alpha$, then the equation
\[
\upF_A=\upF_B+\cci\cdot\d\alpha
\]
describes the curvature of~$A$ in terms of~$B$ and~$\alpha$. 

We may write the curvature as a multiple~$\upF_A=\star\upF_A\omega$ of the volume form, with the function~$\star\upF_A$, using the Hodge operator~$\star$ on~$X$. 

The real~$2$-form~$(\cci/2\pi)\upF_A$ is the first Chern form, which represents the first~(real) Chern class of~$L$. In particular,
\[
\int_X\frac\cci{2\pi}\upF_A=\deg(L).
\]
Conversely, using the above parametrization of the space of unitary connections by the space of real~$1$-forms, it is easy to see that every representative of the first Chern class arises in this way from a unitary connection.

\begin{proposition}\label{prop:harmonic}
If the unitary connection~$A$ has a harmonic curvature form, then the equality
\[
\frac\cci{2\pi}\upF_A=\frac{\deg(L)}{\vol(X)}\omega
\]
holds, and it describes the first Chern form.
\end{proposition}

\begin{proof}
If the connection~$A$ has a harmonic curvature form, then the
curvature form is a~(locally) constant multiple of the volume form,
and conversely. An easy integration determines the constant.
\end{proof}

\subsection{The Weitzenb\"ock formula} 

If~$A$ is a unitary connection on a Hermitian line bundle~$L$, then the complex structure on the Riemann surface~$X$ splits the covariant differential: \hbox{$\d_A=\p_A+\pbar_A$}. 

\begin{remark}\label{rem:holomorphic}
The operator~$\pbar_A$ defines a holomorphic structure on~$L$.
Conversely, a holomorphic structure and the Hermitian metric determine
a unique unitary connection on~$L$, the Chern connection.
\end{remark}

\begin{proposition}[\bf Weitzenb\"ock formula]\label{prop:Weitzen}
\[
  \d^*_A\d^{\phantom*}_A=2\pbarstar_A\pbar_A+\cci\star \upF_A
\]
\end{proposition}

\begin{proof}
Note that
\[
  \d^*_A\d^{\phantom*}_A=\pstar_A\p_A+\pbarstar_A\pbar_A.
\]
One K\"ahler identity
$\pbarstar_A=\cci[\p_A,\star]$ 
implies~$\pbarstar_A=-\cci\star\p_A$ for~$(0,1)$-forms, hence
\[
  \pbarstar_A\pbar_A=-\cci\star\p_A\pbar_A
\] 
on sections. Similarly, the other K\"ahler identity
$\pstar_A=-\cci[\pbar_A,\star]$ leads to
\[
  \pstar_A\p_A=\cci\star\pbar_A\p_A.
\]
Together with the identity~$\upF_A=\pbar_A\p_A+\p_A\pbar_A$, we obtain
\[
  \cci\star \upF_A=\cci\star\upF_A=\cci\star\pbar_A\p_A+\cci\star\p_A\pbar_A,
\]
and from this the formula is derived.
\end{proof}

%%%%%%%%%%%%%%%%%%%%%%%%%%%%%%%%%%%%%%%%%%%%%%%%%%%%%%%%%%%%%%%%%%%%%%

\section{Gauge symmetries}\label{sec:gauge_symmetries}

If~$(X,L)$ is a pair as in Section~\ref{sec:dimension1}, so that~$X$ is a Riemann surface and~$L$ is a Hermitian line bundle over it, then the automorphism group~$\Aut(X,L)$ is the group of pairs~$(f,u)$, where~$f$ is an orientation preserving diffeomorphism of~$X$, and~$u$ is an isometric isomorphism~\hbox{$L\cong f^*L$}. Note that such an isomorphism~$L\cong f^*L$ always exists as long as~$f$ is orientation preserving. Therefore, the automorphism group sits in an extension
\begin{equation}\label{eq:extension}
1\longrightarrow\Map(X,\TT)
\longrightarrow\Aut(X,L)
\longrightarrow\Diff(X)
\longrightarrow1,
\end{equation}
where~$\Diff(X)$ denotes the group of orientation preserving diffeomorphism of the Riemann surface~$X$, and
\[
\calG(X)=\Map(X,\TT)
\] 
is the {\it gauge group} of maps with values in the circle group~$\TT$.

\subsection{The structure of the gauge group}

The gauge group, just as every topological group, sits in an extension
\begin{equation}\label{eq:G_extension}
1\longrightarrow\calG_0(X)
\longrightarrow\calG(X)
\longrightarrow\pi_0\calG(X)
\longrightarrow1,
\end{equation}
where~$\calG_0(X)$ is the component of the constant maps, and~$\pi_0\calG(X)$ is the group of components.

\begin{proposition}
The embedding~$\TT\subseteq\calG_0(X)$ of the circle group, thought of as the subgroup of constant maps, into their component, is a homotopy equivalence. 
\end{proposition}

\begin{proof}
By covering space theory, the map
\[
\Omega^0_X\longrightarrow\calG_0(X),\,f\longmapsto\exp(2\pi\cci f),
\]
is surjective. The kernel is the subgroup of maps with integral values, which can and will be identified with~$\ZZ$. Therefore, the extension
\[
0\longrightarrow\ZZ
\longrightarrow\RR
\longrightarrow\TT
\longrightarrow1,
\]
embeds into the extension
\[
0\longrightarrow\ZZ
\longrightarrow\Omega^0_X
\longrightarrow\calG_0(X)
\longrightarrow1,
\]
to show that there is an isomorphism
\[
\Omega^0_X/\RR\cong\calG_0(X)/\TT,
\]
and the left hand side is a contractible vector space.
\end{proof}

For later purposes, it will be useful to identify a complementary subgroup of~$\TT$ in the component~$\calG_0(X)$. A particularly convenient choice is the image
\[
\calG_0(X,\vol)=\{u\in\calG_0(X)\,|\,u=\exp(2\pi\cci f)\text{ for some }f\text{ with }\int_Xf=0\}
\]
of the complement
\[
\{f\in\Omega^0_X\,|\,\int_Xf=0\}\subseteq\Omega^0_X
\] 
of the space of constant functions.

\begin{remark}
There are also other complements of~$\TT$ in~$\calG_0(X)$, such as
\[
\calG_0(X,x)=\{u\in\calG_0(X)\,|\,u(x)=1\},
\]
but these depend on the choice of a base point~$x$ in~$X$, and they are therefore not canonical. One may, however, have reasons to prefer such a choice.
\end{remark}

\begin{remark}\label{rem:gerbe}
The extension~\eqref{eq:extension} of groups leads to a gerbe on the classifying space~$\BDiff(X)$. If~$L$ happens to be the tangent bundle of~$X$, then the extension~\eqref{eq:extension} is split by the derivative. In general, the isomorphism class of the gerbe depends on the residue class of~$\deg(L)$ modulo the degree of the tangent bundle, which is the Euler characteristic of~$X$.
\end{remark}

The group
\[
\pi_0\calG(X)=[X,\TT]=\upH^1(X;\ZZ)
\]
of components is the group of homotopy classes of maps~$X\to\TT$, and its structure does depend on the topology of~$X$: it is free abelian of rank~$2g(X)$, where~$g(X)$ denotes the genus of~$X$. This implies that the extension~\eqref{eq:G_extension} can be split, but there is no canonical splitting, and we will fortunately never have the need to choose one.

\subsection{The Picard torus}\label{sec:Picard_torus}

Gauge theory provides for geometric models of some classifying spaces related to gauge groups, and this will be explained now.

The gauge group~$\calG(X)$ acts on the space~$\calA(L)$ of unitary connections on~$L$ via conjugation:
\[
\d_{u\cdot A}=u\,\d_Au^{-1}=d_A+u\,\d u^{-1}=d_A-u^{-1}\d u.
\] 

\begin{example}
If~$u=\exp(2\pi\cci f)$ for some real function~$f$, then
\begin{equation}\label{eq:exp_action}
u^{-1}\d u=2\pi\cci\cdot\d f.
\end{equation}
In particular, if~$u$ is constant, then~$u\cdot A=A$, so that the subgroup~$\TT$ of constant maps stabilises all connections.
\end{example}

In any case, we have
\begin{equation}\label{eq:curvature_is_invariant}
\upF_{u\cdot A}=\upF_A.
\end{equation}
In other words, gauge equivalent connections have the same curvature. The converse need not hold. In fact, the set of gauge equivalence classes of unitary connections with a fixed curvature form is a~$2g(X)$-dimensional torus. We will explain this for a particular choice of curvature form: the harmonic one. To do so, let
\[
\calA_\h(L)\subseteq\calA(L)
\]
be the subspace of unitary connections with harmonic curvature. By~\eqref{eq:curvature_is_invariant}, this is~$\calG(X)$-invariant. If a unitary connection~$B$ has harmonic curvature, then so has~$B+\cci\alpha$ if and only if~$\d\alpha=0$. This gives~$\calA_\h(L)$ the structure of a non-empty torsor for the vector space~$\upZ^1_X$ of closed real~$1$-forms. In other words, any choice of such a~$B$ gives rise to an identification
\begin{equation}\label{eq:identification}
\upZ^1_X\cong\calA_\h(L).
\end{equation}
Equation~\eqref{eq:exp_action} shows that~$u=\exp(2\pi\cci f)$ in~$\calG_0(X)$ acts on this space by subtraction of~$2\pi\d f$. This means that the orbit spaces can be identified with certain quotient spaces as follows.

\begin{proposition}
Any choice of a unitary connection~$B$ with harmonic curvature determines, via~\eqref{eq:identification}, an identification
\[
\calA_\h(L)/\calG_0(X)\cong\upZ^1_X/\B^1_X=\upH^1(X;\RR)
\] 
of the orbit space with the first de Rham cohomology of~$X$. Similarly, there result identifications~$\calA(L)/\calG_0(X)\cong\Omega^1_X/\B^1_X$.
\end{proposition}

\begin{remark}
Since the subgroup~$\TT$ of constant maps acts trivially, we could have used one if its complements in~$\calG_0(X)$ throughout.
\end{remark}

The residual action by~$\pi_0\calG(X)=\upH^1(X;\ZZ)$ shows the following.

\begin{proposition}
Any choice of a unitary connection~$B$ with harmonic curvature determines an identification of the orbit space~$\calA_\h(L)/\calG(X)$ with a~$2g(X)$-dimensional torus.
\end{proposition}

In contrast, the orbit space~$\calA(L)/\calG(X)$ only has the homotopy type of a~$2g(X)$-dimensional torus.

\begin{definition}
The orbit space~$\calA_\h(L)/\calG(X)$ will be called the {\it Picard torus}, and it will be denoted by~$\Pic(X,L)$.
\end{definition}

Note that this torus does no come with a preferred base-point; we had to choose the base point~$[B]$. And, {\it a fortiori}, it does not come with a group structure. 

\begin{remark}
The Picard torus~$\Pic(X,L)$ just defined can be canonically identified with one single component of what one would usually define to be the Picard group~$\Pic(X)$ of~$X$: the one corresponding to~$L$. The identification sends a holomorphic line bundle to its Chern connection, compare with Remark~\ref{rem:holomorphic}.
\end{remark}

%%%%%%%%%%%%%%%%%%%%%%%%%%%%%%%%%%%%%%%%%%%%%%%%%%%%%%%%%%%%%%%%%%%%%%

\section{The vortex equations}\label{sec:equations}

Let again~$(X,L)$ be a Riemann surface~$X$ with a Hermitian line bundle~$L$ on it. 

\begin{definition}
The~{\it vortex equations} for a section~$\phi$ of~$L$ and a unitary
connection~$A$ on~$L$ are
\begin{eqnarray}
   \pbar_A\phi&=&0\label{eq:ve1}\\
   \upF_A&=&\frac\cci2(|\phi|^2-1)\omega\label{eq:ve2}.
\end{eqnarray}
\end{definition}
Since it will be convenient to work with real~$2$-forms rather than with imaginary~$2$-forms, the second vortex equation will also be written in form
\[
   \cci\upF_A=\frac12(1-|\phi|^2)\omega.
\]
And, later on, we will also have occasion to replace~\eqref{eq:ve2} with the more general~$\tau$-vortex equation
\begin{eqnarray}
   \upF_A&=&\frac\cci2(|\phi|^2-\tau)\omega\label{eq:ve2tau}
\end{eqnarray}
which involves another (real) parameter~$\tau$, the {\it Taubes parameter}.

The vortex equations can be understood as a 2-dimensional version of the Seiberg-Witten equations. They have been studied by Ginzburg, Landau, Jaffe, Taubes, Bradlow, Garc\'\i a-Prada and many others. See~\cite{JaffeTaubes}, \cite{Taubes:CommMathPhys}, \cite{Bradlow}, \cite{GarciaPrada:CommMathPhys}, \cite{GarciaPrada:LMS}, \cite{GarciaPrada:Survey}, \cite{Jost}, \cite{Taues:GW=>SW}
 and the references therein.

\subsection{Symmetries}

The gauge group~$\calG(X)$ acts on the set of pairs~$(\phi,A)$ via
\[
u(\phi,A)=(uA,u\phi),
\]
where the action~$uA$ is explained in Section~\ref{sec:Picard_torus}, and~$u\phi$ is just the usual pointwise multiplication of a section by a function. The space~$\calV(X,L)$ of solutions to the vortex equations is easily checked to be invariant: The first vortex equation~\eqref{eq:ve1} implies
\[
\pbar_{uA}(u\phi)=u\pbar_A(u^{-1}u\phi)=u0=0,
\]
and the second vortex equation~\eqref{eq:ve2} for~$(uA,u\phi)$ follows from the one for~$(A,\phi)$ using~\hbox{$\upF_{uA}=\upF_A$} and~$|u\phi|=|u|\cdot|\phi|=|\phi|$.

If a function~$u$ stabilises a connection~$A$, then~$u^{-1}\d u=0$, which is the case if and only if the function~$u$ is~(locally) constant. However a constant function~$u\not=1$ stabilises a section~$\phi$ if and only if~$\phi=0$. Therefore, the stabiliser of a pair~$(\phi,A)$ is the entire group~$\TT$ if~$\phi=0$, and trivial else. If we define
\[
\calV^\times(X,L)=\{(\phi,A)\in\calV(X,L)\,|\,\phi\not=0\},
\]
then~$\TT$ acts freely on~$\calV^\times(X,L)$, and the projection to the orbit space
\[
\calM^\times(X,L)=\calV^\times(X,L)/\TT
\]
defines a principal~$\TT$-bundle over~$\calM^\times(X,L)$.

 \subsection{Properties of solutions}

Let us set
\[
  \tau_0=\frac{4\pi\deg(L)}{\vol(X)}.
\]
The factor~$4\pi$ should be thought of as the volume of the Riemann sphere of radius~$1$.

\begin{proposition}\label{prop:length}
Every solution~$(\phi,A)$ of the vortex equations satisfies
\[
  \|\phi\|^2=(1-\tau_0)\cdot\vol(X).     
\]
\end{proposition}

\begin{proof}
This result is easily obtained by integrating the second vortex equation~\eqref{eq:ve2}.
\end{proof}

\subsection{Spaces of solutions}

If~$\tau_0>1$, then Proposition~\ref{prop:length} shows that there can be no solutions to the vortex equations, so that the solution space~$\calV(X,L)$ is empty.

If~$\tau_0=1$, then Proposition~\ref{prop:length} shows that~$\phi=0$ is necessary for a solution. 
Proposition~\ref{prop:harmonic} implies for solutions with~$\phi=0$ that the second vortex equation~\eqref{eq:ve2} is satisfied if and only if the curvature form of~$A$ is harmonic. Thus, in this case, the moduli space is a torus of dimension~$2g(X)$ by the results described in Section~\ref{sec:Picard_torus}.

If~$\tau_0<1$, then every solution has~$\phi\not=0$. In~\cite{Bradlow}, Bradlow has shown
that for every section~$\phi$ there is a unique unitary connection~$A$ satisfying the second
vortex equation. Consequently, the moduli space is identified with the space of
effective divisors of degree~$d=\deg(L)$, i.e. with the~$d$-fold
symmetric power of~$X$. This space maps to the Picard torus with projective spaces
as fibres.

The following result is well-known.
  
\begin{proposition}
On each Riemann surface~$X$ there are only finitely many isomorphism classes of line bundles~$L$ such that the vortex equations on~$(X,L)$ have a solution.
\end{proposition} 

\begin{proof}
On the one hand, the degree~$\deg(L)$ of~$L$ has to be large enough such that the bundle has a non-zero holomorphic section. On the other hand, Proposition~\ref{prop:length} implies that there are no solutions in case the degree~$\deg(L)$ of~$L$ is larger than~$\vol(X)/4\pi$.
\end{proof}

%%%%%%%%%%%%%%%%%%%%%%%%%%%%%%%%%%%%%%%%%%%%%%%%%%%%%%%%%%%%%%%%%%%%%%

\section{The vortex map}\label{sec:map}

As in Section~\ref{sec:dimension1}, let~$(X,L)$ be a pair consisting of a Riemann surface~$X$ together with a Hermitian line bundle~$L$. 

\subsection{The based vortex map}

We will use a unitary base connection~$B$ on~$L$ with harmonic curvature~$\upF_B$ to parametrize the space of unitary connections by the space~$\Omega^1_X$ of real~$1$-forms. The harmonic assumption is inessential; it is only used to write~$\tau_0/2$ instead of~$\cci\star\upF_B$ sometimes, see Section~\ref{sec:harmonic_used}.

\begin{definition}
The {\it vortex map for~$(X,L)$ based at~$B$} is the~$\calG(X)$-equivariant map
\begin{equation}
  v_B=v_B(X,L)\colon
   \Omega^0_X(L)\oplus\Omega^1_X
   \longrightarrow
   \Omega^{01}_X(L)\oplus\Omega^2_X\oplus\mathcal{H}^1_X\oplus\overline{\Omega}^0_X
\end{equation}
which sends~$(\phi,\alpha)$ to
\[
  (\pbar_{B+\cci \alpha}\phi,\cci\upF_{B+\cci \alpha}-\frac12(1-|\phi|^2)\omega,\h\alpha,[\d^*\alpha]),
\]
where~$\h\alpha$ is the image of~$\alpha$ under the projection to the space~$\mathcal{H}^1_X$ of harmonic~$1$-forms, and~$[\d^*\alpha]$ is the class of~$\d^*\alpha$ modulo constants in~$\overline{\Omega}^0_X=\Omega^0_X/\RR$. More generally, the~{\it~$\tau$-vortex map~$v_B^\tau(X,L)$ for~$(X,L)$ based at~$B$} is defined by building in the Taubes parameter~$\tau$ according to vortex equation~\eqref{eq:ve2tau}.
\end{definition}

Implicitly, will shall pass to~$L^2_4$- and~$L^2_3$-completions in the source and target, respectively. Then the vortex map for the pair~$(X,L)$ based at~$B$ is a continuous map between Hilbert spaces. 

\subsection{Variation of the base connection}

In order to eliminate the dependence on the choice of the unitary connection~$B$ with harmonic curvature form, we may consider the~$\calG(X)$-equivariant map
\[
   \calA_\h(L)\times\Big[\Omega^0_X(L)\oplus\Omega^1_X\Big]
   \longrightarrow
   \calA_\h(L)\times\Big[\Omega^{01}_X(L)\oplus\Omega^2_X\oplus\mathcal{H}^1_X\oplus\overline{\Omega}^0_X\Big]
\]
over~$\calA_\h(L)$, which sends~$(B,(\phi,\alpha))$ to~$(B,v_B(\phi,\alpha))$. 

The group~$\calG(X)$ is not a compact Lie group. It does not even have the homotopy type of a finite CW complex, unless~$X$ has vanishing genus. Since it is technically easier to work with compact Lie groups, we shall handle the information encoded in the~$\calG(X)$-equivariance of the vortex map following Bauer~\cite{Bauer:survey} again: We shall keep the action of the compact subgroup~$\TT$ of~$\calG(X)$, so that everything is sight is still~$\TT$-equivariant. And, we incorporate the rest of the action by working parametrized over a classifying space for a complement of~$\TT$, for which a gauge theoretical model has been described in Section~\ref{sec:Picard_torus}. Note that~$\TT$ acts trivially on connections, but the residual action of~$\calG(X)/\TT$ is free.

After~$L^2$-completions, this leads to a~$\TT$-equivariant map 
\begin{equation}\label{eq:vortex_map}
   \calA_\h(L)\underset{\calG(X)}{\times}\Big[\Omega^0_X(L)\oplus\Omega^1_X\Big]
   \longrightarrow
   \calA_\h(L)\underset{\calG(X)}{\times}\Big[\Omega^{01}_X(L)\oplus\Omega^2_X\oplus\mathcal{H}^1_X\oplus\overline{\Omega}^0_X\Big]
\end{equation}
of Hilbert bundles over the Picard torus~$\calA_\h(L)/\calG(X)=\Pic(X,L)$.

\begin{definition}
The map~\eqref{eq:vortex_map}, which is~$\TT$-equivariant and parametrized over the Picard torus~$\Pic(X,L)$, is called the {\it vortex map} of~$(X,L)$, and it will be denoted by~\hbox{$v=v(X,L)$}.
\end{definition}

%%%%%%%%%%%%%%%%%%%%%%%%%%%%%%%%%%%%%%%%%%%%%%%%%%%%%%%%%%%%%%%%%%%%%%

\section{Analytic estimates}\label{sec:estimates}

The purpose of this section is to provide for the necessary estimates to show that the vortex map can be used to define a stable homotopy class.

\begin{theorem}\label{thm:estimates}
The vortex map is a continuous map between Hilbert spaces. Preimages of bounded sets are bounded, and the map can be written as a sum of a linear Fredholm operator and of a compact map.
\end{theorem}

The proof of this result takes the rest of this section.

\subsection{An {\it a priori} bound} 

Given a real~$1$-form~$\alpha$ let~$A$ be the
connection~$B+\cci\alpha$. Since~$A$ is also unitary,
\[
  \Delta|\phi|^2 
  =\d^*\d\langle\phi,\phi\rangle
  = 2\langle\d^*_A\d^{\phantom*}_A\phi,\phi\rangle 
  - 2\langle\d_A\phi,\d_A\phi\rangle
  \le 2\langle\d^*_A\d^{\phantom*}_A\phi,\phi\rangle.
\]
Inserting the Weitzenb\"ock formula~(Proposition~\ref{prop:Weitzen}), this
leads to the pointwise estimate
\begin{equation}\label{eq:pointwise_estimate}
  \Delta|\phi|^2 
  \le 
  4\langle\pbarstar_A\pbar_A\phi,\phi\rangle + 
  2\langle\cci\star \upF_A\phi,\phi\rangle.
\end{equation}

Consequently, if~$\phi$ and~$A$ solve the~$\tau$-vortex equations~(\ref{eq:ve1})
and~(\ref{eq:ve2tau}), it follows that
\[
   \Delta|\phi|^2\le(\tau-|\phi|^2)|\phi|^2.
\]
If~$\phi$ is non-zero, the maximum of~$|\phi|^2$ is non-zero, namely
$\|\phi\|^2_\infty$, and at the maximum one has~$\Delta|\phi|^2\ge
0$. It follows that~$\|\phi\|^2_\infty\le\tau$ so that the~$\phi$ are
uniformly bounded.

More generally, assume that there are~$L^2_3$-bounds on
\begin{equation}
  \label{eq:notation}    
  \psi = \pbar_A\phi 
  \qquad\mathrm{and}\qquad
  b = \cci\star \upF_A-\frac12(\tau-|\phi|^2).
\end{equation}
The aim is to show that there are~$L^2_4$-bounds on~$\phi$ and~$\alpha$.

\subsection{Uniformly bounding the section~\texorpdfstring{$\phi$}{phi}} 

One may insert the new notation~(\ref{eq:notation}) into the pointwise
estimate~(\ref{eq:pointwise_estimate}). Using
\[
  \pbarstar_A\psi=(\pbar_{B+\cci\alpha})^*\psi = \pbarstar_B\psi + ((\cci\alpha)^{0,1})^*\psi,
\]
and the notation~$x\lesssim y$ if~$x\leq Ky$ for some constant~$K$ which is independent of~$x$ and~$y$, but may otherwise vary from inequality to inequality, one obtains
\begin{eqnarray*}
  \Delta|\phi|^2 
  &\lesssim& 4|\pbarstar_A\psi||\phi| +
  (\tau+2b)|\phi|^2 - |\phi|^4\\
  &\lesssim& 
  \left(\|\pbarstar_B\psi\|_\infty+\|\alpha\|_\infty\|\psi\|_\infty\right)|\phi| 
  + \left(1+\|b\|_\infty\right)|\phi|^2 - |\phi|^4.
\end{eqnarray*}
The Sobolev embedding of~$L^2_2$ into~$C^0$
gives
\[
  \|\pbarstar_B\psi\|_\infty \lesssim
  \|\pbarstar_B\psi\|_{2,2}\lesssim\|\psi\|_{2,3},
\]
and therefore
\begin{equation}
  \label{eq:estimate_for_section}
  \Delta|\phi|^2 
  \lesssim 
  \left(1+\|\alpha\|_\infty\right)\|\psi\|_{2,3}|\phi| 
  + \left(1+\|b\|_{2,3}\right)|\phi|^2 - |\phi|^4.
\end{equation}

\subsection{Uniformly bounding the form~\texorpdfstring{$\alpha$}{alpha}}\label{sec:harmonic_used}

If~$p>2$, then there is the Sobolev embedding of~$L^p_1$ into~$C^0$. Using the
elliptic inequality, and~$\d^*\alpha=0$, one finds
\[
  \|\alpha\|_\infty 
  \lesssim \|\alpha\|_{p,1} 
  \lesssim \|\d\alpha\|_p.
\]
One has~$\cci\star\upF_A= \cci\star\upF_B-\star\d\alpha$. In the case when the curvature~$\upF_B$ of the base connection~$B$ is harmonic, then we also have~$\cci\star \upF_B=\tau_0/2$. This gives
\[
  \star\d\alpha = \frac12|\phi|^2-\frac12(\tau-\tau_0)-b
\]
and therefore
\begin{eqnarray*}
  \|\d\alpha\|_p 
  &\lesssim&\||\phi|^2\|_p+\|\tau-\tau_0\|_p+\|b\|_p\\
  &\lesssim&\|\phi\|^2_\infty+\|\tau-\tau_0\|_p+\|b\|_{2,3}.
\end{eqnarray*}
The preceding two inequalities lead to the estimate
\begin{equation}
  \label{eq:estimate_for_form}
  \|\alpha\|_\infty \lesssim \|\phi\|^2_\infty+\|\tau-\tau_0\|_p+\|b\|_{2,3}.
\end{equation}

\subsection{Uniformly bounding both the section \texorpdfstring{$\phi$}{phi} and the form \texorpdfstring{$\alpha$}{alpha}} 

Inserting the previous inequality into~(\ref{eq:estimate_for_section})
and evaluating at a point where~$\phi$ is maximal, one obtains a
uniform bound for~$\phi$.  Then~(\ref{eq:estimate_for_form}) gives a
uniform bound for~$\alpha$.

\subsection{Bootstrapping} 

One may now use the uniform bounds on~$\phi$ and~$\alpha$, the
assumed~$L^2_3$-bounds on~$\psi$ and~$b$, and the vortex equations to show
that there are in fact~$L^2_4$ bounds on~$\phi$ and~$\alpha$.  

First use the H\"older multiplication from~$L^{2p}\times L^{2p}$ to
$L^p$ for some~$p>2$ to get~$L^p_1$-bounds as follows. By definition,
\[
  \|\phi,\alpha\|^p_{p,1}=\|\phi,\alpha\|^p_p+  \|\pbar_B\phi,\d\alpha\|^p_p.
\]
The first term on the right hand side is bounded by the uniform
bounds. For the other one, the notation~(\ref{eq:notation}) gives
\begin{eqnarray*}
  \|\phi,\alpha\|^p_{p,1} 
  & = &
  \|\phi,\alpha\|^p_p + 
  \|\psi-\cci\alpha^{0,1}\phi,
  \frac12|\phi|^2-\frac12(\tau-\tau_0)-b\|^p_p\\
  &\le &
  \|\phi,\alpha\|^p_p + 
  \|\psi,b\|^p_p + 
  \|\alpha^{0,1}\phi\|^p_p+\frac12\||\phi|^2\|^p_p+\frac12\|\tau-\tau_0\|^p_p.
\end{eqnarray*}
Now the second term is bounded by assumption and the last term is
constant. For the term in between use the above-mentioned
multiplication to get
\[
  \|\alpha^{0,1}\phi\|^p_p+\frac12\||\phi|^2\|^p_p
  \lesssim \|\alpha\|^p_{2p}\|\phi\|^p_{2p}+\frac12\|\phi\|^{2p}_{2p}.
\] 
Thus indeed the~$L^{2p}$-bounds on~$\phi$ and~$\alpha$ give~$L^p_1$-bounds
on them.

Similarly use the Sobolev multiplication on~$L^p_1$ for~$p>2$, which yields 
\[
  \|\alpha^{0,1}\phi\|^p_{p,1}+\frac12\||\phi|^2\|^p_{p,1}
  \lesssim \|\alpha\|^p_{p,1}\|\phi\|^p_{p,1}+\frac12\|\phi\|^{2p}_{p,1},
\] 
to get~$L^p_2$-bounds. The embedding of~$L^p_2$ into~$L^2_2$ then
gives~$L^2_2$ bounds. And using the Sobolev multiplication on~$L^2_k$
for~$k>1$ one gets~$L^2_3$- and finally~$L^2_4$-bounds as desired.

%%%%%%%%%%%%%%%%%%%%%%%%%%%%%%%%%%%%%%%%%%%%%%%%%%%%%%%%%%%%%%%%%%%%%%

\section{Stable cohomotopy invariants}\label{sec:invariants}

We have already defined the vortex map for a pair~$(X,L)$ in Section~\ref{sec:map} and studied some of its analytic properties in Section~\ref{sec:estimates}. Now we will distill this information into a stable homotopy class.

%%%%%%%%%%%%%%%%%%%%%%%%%%%%%%%%%%%%%%%%%%%%%%%%%%%%%%%%%%%%%%%%%%%%%%

\subsection{The Bauer-Furuta construction}\label{sec:BauerFuruta}

The Bauer-Furuta construction, see~\cite[Section~2]{BauerFuruta} and~\cite[Section~2]{Bauer:survey}, which generalizes earlier work of \v{S}varc, see~\cite{Svarc}, allows us to pass from certain non-linear perturbations of linear Fredholm operators to stable homotopy classes of maps. For later reference, we state their result in the following form.

\begin{theorem}\label{thm:Bauer}
A continuous map between Hilbert spaces with the property that preimages of bounded sets are bounded, and that can be written as a sum of a linear Fredholm operator~$\ell$ and of a compact map, defines a stable homotopy class~$\rmS^\ell\to\rmS^0$ between (finite-dimensional) spheres. (If the index of~$\ell$ is negative, then the one-point-compactification~$\rmS^\ell$ has to be interpreted as a spectrum.) More generally, an equivariant family of such maps between~$G$-Hilbert space bundles over finite CW complexes~$B$ defines a class that lives in
\begin{equation}\label{eq:group}
[\rmS^\ell_B,\rmS^0_B]^G_B
\end{equation}
the~$0$-th~$G$-equivariant parametrized stable cohomotopy over~$B$.
\end{theorem}

Often, the group \eqref{eq:group} can also be interpreted in a non-parametrized way, by passing to the Thom spectrum of the index bundle. This will be explained later.

%%%%%%%%%%%%%%%%%%%%%%%%%%%%%%%%%%%%%%%%%%%%%%%%%%%%%%%%%%%%%%%%%%%%%%

\subsection{The linearisation of the vortex map}\label{sec:index}

The Taylor expansion at~$(0,0)$ of the vortex map~$v_B=v_B(X,L)$ for~$(X,L)$ based at~$B$ is
\begin{eqnarray*}
v_B(\phi,\alpha)&=&(0,\cci\upF_B-\frac12\omega,0,0)\\
&+&(\pbar_{B}\phi,-\d\alpha,\h\alpha,[\d^*\alpha])\\
&+&(\cci\cdot\d\alpha\cdot\phi,\frac12|\phi|^2\omega,0,0).
\end{eqnarray*}
Let us work out the index of the linearization
\[
(\alpha,\phi)\longmapsto(\pbar_{B}\phi,-\d\alpha,\h\alpha,[\d^*\alpha]),
\]
which can be decomposed as a real part and a complex linear part.

On the one hand, the real operator 
\[
\alpha\longmapsto(-\d\alpha,\h\alpha,[\d^*\alpha]),
\]
is elliptic with real index~$-1$. In fact, this linear map is precisely the restriction~$v_B(X,L)^\TT$ of the vortex map~$v_B(X,L)$ to the~$\TT$-fixed point set, and we will shall pause to discuss this in some detail now.

\begin{proposition}
The fixed point map~$v_B(X,L)^\TT$ is injective but not surjective.
\end{proposition}

\begin{proof}
If~$(0,\alpha)$ is a fixed point, then it is sent to~$(0,-\d\alpha,\h\alpha,[\d^*\alpha])$. If that image is zero, then we have~$\d\alpha=0$ and~$\d^*\alpha=0$, so that~$\alpha$ is harmonic. Therefore~$\alpha=\h\alpha=0$. This shows that~$v_B^\TT$ is injective. Since the index is negative, it cannot be surjective: the image is the orthogonal complement of the~(1-dimensional) space of harmonic~$2$-forms. 
\end{proof}

If we vary the pair~$(X,L)$ in a family, then the index of the real part of the linearisation of the vortex map is the negative of a line bundle over the base of the family. This is a real line bundle, the so-called {\it Hodge bundle}~$\beta=\beta^2(X)$ of harmonic~$2$-forms in the fibres. It only depends on the variation of~$X$, of course. This bundle is flat, hence trivial as soon as the fundamental group of the base vanishes, but also in many other cases. The universal example lives over~$\BDiff(X)$, where the fundamental group is the mapping class group of the surface~$X$. Because of our assumption that~$\Diff(X)$ is the group of orientation preserving diffeomorphisms, the Hodge bundle for the families considered here will always be trivial, and a nowhere-zero section is given by the family of volume forms.

On the other hand, the complex operator
\[
  \pbar_B:\Omega^0_X(L)\longrightarrow\Omega^{01}_X(L) 
\]
is also elliptic, and its complex index is given by the well-known Riemann-Roch formula as~\hbox{$\deg(L)+1-g(X)$}. 

As the class of~$B$ varies over the Picard torus, the Chern classes of the index bundle, the {\it Riemann-Roch bundle~$\rho=\rho(X,L)$}, can be computed using the Grothendieck-Riemann-Roch theorem for the push-forward of the Poincar\'e line bundle on~$X\times\Pic(X,L)$, see~\cite{ACGH} for example. The result is
\begin{equation}
\c_k(\rho)=(-1)^k\frac\theta{k!},
\end{equation}
where~$\theta$ is the class of the theta divisor in~$\upH^2(\Pic(X,L);\ZZ)$. For the purposes of this text, it will be sufficient to know that the~$\theta$-divisor is a translate of the image of the product
\[
\Sym^{g(X)-1}(X)\subseteq\Pic(X),
\]
where more generally the~$n$-th symmetric product~$\Sym^n(X)$ maps to the component~$\Pic^n(X)\subseteq\Pic(X)$ of line bundles of degree~$n$ via
\[
[p_1,\dots,p_n]\longmapsto\calO(p_1)\otimes\dots\otimes\calO(p_n).
\]
In particular, for~$g(X)=0$, the~$\theta$-divisor is empty, and for~$g(X)=1$, it is a point.

%%%%%%%%%%%%%%%%%%%%%%%%%%%%%%%%%%%%%%%%%%%%%%%%%%%%%%%%%%%%%%%%%%%%%%

\subsection{Stable homotopy theory invariants}

The general construction reviewed in Section~\ref{sec:BauerFuruta} above, together with the specific properties of the vortex map proven in Section~\ref{sec:estimates}, lead to the first main result of this paper.

\begin{theorem}\label{thm:vortex_define_stable}
  For each pair~$(X,L)$, consisting of a Riemann surface~$X$ and a Hermitian line bundle~$L$, the vortex map defines a stable homotopy class
\[
v(X,L)\in[\rmS^{\rho}_{\Pic(X,L)},\rmS^1_{\Pic(X,L)}]^\TT_{\Pic(X,L)}
\]
in the group of~$\TT$-equivariant stable homotopy classes of maps parametrized by the Picard torus~$\Pic(X,L)$.
\end{theorem}

Since the Hodge bundle~$\beta^2(X)$ over the Picard torus~$\Pic(X,L)$ is always trivial, a standard base change adjunction, see for example~\cite[II.1.3]{Crabb+James}, shows that the group in the theorem can be identified with the group
\[
\pi^1_\TT(\Pic(X,L)^{\rho}),
\]
where the notation~$\Pic(X,L)^{\rho}$ refers to the Thom spectrum of the Riemann-Roch bundle~\hbox{$\rho=\rho(X,L)$} over the Picard torus~$\Pic(X,L)$, at least if that bundle can be represented by an actual vector bundle minus a trivial bundle, so that the Thom spectrum is a (de)suspension of a Thom space.

%%%%%%%%%%%%%%%%%%%%%%%%%%%%%%%%%%%%%%%%%%%%%%%%%%%%%%%%%%%%%%%%%%%%%%

\section{Group actions and families}\label{sec:actions_and_families}

The preceding theory that culminated in Theorem~\ref{thm:vortex_define_stable} has extensions to both the cases where compact Lie groups act on the situation, and where one considers families. This can be set up with only minor changes to the author's previous work in the context of the Bauer-Furuta invariants, see~\cite{Szymik:thesis}, \cite{Szymik:stable}, \cite{Szymik:families}, and~\cite{Szymik:Galois}. Therefore, the form and function of these extensions in the context of the vortex equations will only be sketched here.

\subsection{Group actions}\label{sec:group_action}

If a compact Lie group~$G$ acts on~$X$ preserving a complex
line bundle~$L$, then there is an extension~$\GG$ of~$G$ by~$\TT$ such that the
homomorphism~\hbox{$G\to\Diff(X)$} lifts to a
homomorphism~\hbox{$\GG\to\Aut(X,L)$}. As explained in~\cite{Szymik:thesis} and~\cite{Szymik:Galois}, in a situation which corresponds to the case~$g(X)=0$, there is a~$\GG$-equivariant stable homotopy class which lives in~$\pi^1_\GG(\rmS^\rho)$ and maps to the vortex class under the forgetful map
\[
\pi^1_\GG(\rmS^\rho)\longrightarrow\pi^1_\TT(\rmS^\rho).
\]
In the case~$g(X)\not=0$, the Picard torus is easy to build in, replacing the sphere~$\rmS^\rho$ by the Thom space~$\Pic(X,L)^\rho$ of the index bundle. This lead to forgetful maps
\[
[\rmS^{\rho}_{\Pic(X,L)},\rmS^1_{\Pic(X,L)}]^\GG_{\Pic(X,L)}
\longrightarrow
[\rmS^{\rho}_{\Pic(X,L)},\rmS^1_{\Pic(X,L)}]^\TT_{\Pic(X,L)}
\]
and
\[
\pi^1_\GG(\Pic(X,L)^{\rho})
\longrightarrow
\pi^1_\TT(\Pic(X,L)^{\rho}),
\]
respectively.

\subsection{Families}

As explained in~\cite{Szymik:families}, there are characteristic cohomotopy classes for families with fibre~$(X,L)$ over compact bases~$B$, and a universal characteristic class which lives in the group~$\pi^1_\TT(\BAut(X,L)^\rho)$. 

A typical class of examples of families arises from group actions, and it is instructive to spell this out in detail.

\begin{example}
In the case of a group action~$\GG\to\Aut(X,L)$, as in the preceding Section~\ref{sec:group_action}, the twisted bundle construction defines a family with fiber~$(X,L)$ over the classifying space~$\BGG$, and the universal characteristic class for families can be restricted from the Thom spectrum~$\BAut(X,L)^\rho$ to the Thom spectrum~$\BGG^\rho$. In this way, we obtain a class in~\hbox{$\pi^1_\TT(\BGG^\rho)$}.
\end{example}

The family invariants of the preceding example and the equivariant invariants from Section~\ref{sec:group_action} can be compared by means of a common unification. This will be explained next.

\subsection{A common unification}

There is a universal equivariant class in~\hbox{$\pi^1_\GG(\BAut(X,L)^\rho)$} which maps to both of the classes mentioned above under the obvious maps. The diagram
\begin{center}
  \mbox{ \xymatrix@C=-20pt{ &
    \pi^1_\GG(\BAut(X,L)^\rho)\ar[dl]\ar[dr] &
    \\ \pi^1_\GG(\rmS^\rho) & &
    \pi^1_\TT(\BAut(X,L)^\rho)\ar[d]\\
    {\phantom{\pi^1_\TT(\BAut(X,L)^\rho)}} & &
    \pi^1_\TT(\BGG^\rho)  } }
\end{center}
summarizes the situation, see~\cite{Szymik:families}.

%%%%%%%%%%%%%%%%%%%%%%%%%%%%%%%%%%%%%%%%%%%%%%%%%%%%%%%%%%%%%%%%%%%%%%

\section{The spherical case}\label{sec:g=0}

If~$X$ is a Riemann surface such that the genus~$g(X)$ vanishes, then the topology of the surface~$X$ is that of a~2-sphere~$\rmS^2$. In this case, the situation simplifies considerably, because the group~$\calG(\rmS^2)$ is connected and the Picard torus is a point. Let us write~$(\rmS^2,d)$ for~$(\rmS^2,L)$ when~$d$ is the degree~$\deg(L)$ of~$L$. 

The Riemann-Roch formula implies that the rank of~$\rho$ is~$d+1$, so that the vortex class is an element
\[
v(\rmS^2,d)\in[\rmS^{(d+1)\CC},\rmS^1]^\TT.
\]
We will give some numerical information about these groups right away.

\subsection{Computations}

While the precise structure of the groups~$[\rmS^{(d+1)\CC},\rmS^1]^\TT$ in which the vortex class lives, is only known in some cases, compare Remark~\ref{rem:structure}, there is the following general result.

\begin{proposition}\label{prop:finite}
The group~$[\rmS^{(d+1)\CC},\rmS^1]^\TT$ is always finite, and trivial for negative~$d$.
\end{proposition}

\begin{proof}
If~$d$ is negative, then~$-(d+1)\geqslant0$, and there is an isomorphism
\[
[\rmS^{(d+1)\CC},\rmS^1]^\TT\cong
[\rmS^0,\rmS^{1-(d+1)\CC}]^\TT
\]
which shows that this group is zero.

If~$d$ is not negative, then there is a cofibration sequence
\[
\rmS(\CC^{d+1})_+
\longrightarrow D(\CC^{d+1})_+
  \longrightarrow \rmS^{(d+1)\CC}
  \longrightarrow \Sigma \rmS(\CC^{d+1})_+
  \longrightarrow \rmS^1
 \]
with~$D(\CC^{d+1})_+\simeq_\TT \rmS^0$. It induces a long exact sequence
\begin{equation}\label{eq:les}
  [\rmS^0,\rmS^1]^{\TT}
  \longleftarrow [\rmS^{(d+1)\CC},\rmS^1]^{\TT}
  \longleftarrow [\Sigma \rmS(\CC^{d+1})_+,\rmS^1]^{\TT}
  \longleftarrow [\rmS^1,\rmS^1]^{\TT}.
\end{equation}
Since~$\TT$ acts freely on the unit sphere~$\rmS(\CC^{d+1})$, there is also an isomorphism
\[
[\Sigma \rmS(\CC^{d+1})_+,\rmS^1]^\TT\cong
[\rmS(\CC^{d+1})_+,\rmS^0]^\TT\cong
[\CC P^d_+,\rmS^0],
\]
and this group is isomorphic to the direct sum of a copy of the integers~$\ZZ$, which comes from the base point, and to a copy of the group~$[\CC P^d,\rmS^0]$, which is easily seen to be finite using Serre's finiteness of the stable stems in positive dimensions.

Clearly, the group~$[\rmS^0,\rmS^1]^{\TT}$ is zero. The group~\hbox{$[\rmS^1,\rmS^1]^{\TT}\cong[\rmS^0,\rmS^0]^{\TT}$} is (canonically) isomorphic to~$\ZZ$. In the long exact sequence induced by the cofibration sequence, it injects into the group~$[\rmS(\CC^{d+1})_+,\rmS^0]^{\TT}$, as can be seen by looking at the forgetful maps.
\begin{center}
  \mbox{ 
  \xymatrix{
[\rmS(\CC^{d+1})_+,\rmS^0]^{\TT}\ar[d]_{\text{forget}} & [\rmS^0,\rmS^0]^{\TT}\ar[l]\ar[d]^{\text{forget}}\\
[\rmS(\CC^{d+1})_+,\rmS^0] & [\rmS^0,\rmS^0]\ar[l]
    } 
    }
\end{center}
The arrow on the right is an isomorphism, and the arrow on the bottom is injective, since there is a (non-equivariant) section of the map from~$\rmS(\CC^{d+1})$ to the singleton. This shows that the map on the top is injective as well, as claimed.
\end{proof}

\begin{remark}\label{rem:structure}
It is not hard to determine the structure of the finite groups~$[\CC P^d,\rmS^0]$ up to isomorphism for small values of~$d$: For~$d=0$ the group~$[\CC P^0,\rmS^0]=[*,\rmS^0]$
is trivial, while for~\hbox{$d=1$} one is concerned with the group~$[\CC
P^1,\rmS^0]=[\rmS^2,\rmS^0]$ which has order~$2$. For~$d=2$ the group~$[\CC
P^2,\rmS^0]$ is again trivial. 
\end{remark}

\subsection{An integral invariant}

Let us review a way to extract integral information from the vortex classes, following~\cite{Bauer:survey}. The starting point is the isomorphism
\[
[\rmS(\CC^{d+1})_+,\rmS^0]^{\TT}\cong[\CC P^d_+,\rmS^0],
\]
which already appeared in the course of the proof of Proposition~\ref{prop:finite}. The right hand side maps to
\[
[\CC P^d_+,\upH\ZZ]=\upH^0(\CC P^d;\ZZ)=\ZZ,
\]
using the unit~$\rmS^0\to\upH\ZZ$ of the integral Eilenberg-MacLane spectrum~$\upH\ZZ$.

Now we look at the long exact sequence~\eqref{eq:les} again. Since the group on the left is zero, every element in the group~$[\rmS^{(d+1)\CC},\rmS^1]^{\TT}$ can be lifted to an element in
the group~\hbox{$[\rmS(\CC^{d+1})_+,\rmS^0]^{\TT}$}. However, the lift will depend on the choice of a null-homotopy. Therefore, we have to make a choice, unless we have a preferred choice of null-homotopy. The set of possible choices is a torsor for the image of the group~\hbox{$[\rmS^1,\rmS^1]^{\TT}=\ZZ$} which is mapped injectively, as we have already seen in the course of the proof of Proposition~\ref{prop:finite}.

\begin{remark}
The traditional way to derive numerical invariants from such situations is to integrate over~$\calM^\times(X,L)$ the top power of the characteristic~$2$-form which is associated to the principal~$\TT$-bundle~$\calV^\times(X,L)\to\calM^\times(X,L)$.
\end{remark}

\subsection{Families of spheres}

So far we have only considered single 2-spheres. Now we shall turn our attention to families of 2-spheres, in particular to the universal family. 

The group~$\Diff(\rmS^2)$ of orientation preserving diffeomorphisms of a 2-sphere is homotopy equivalent to~$\SO(3)$, by Smale's theorem~\cite{Smale}. We have
\[
\upH^*(\BSO(3);\ZZ)\cong\ZZ[e,p_1]/(2e),
\]
and the total space of the universal family of 2-spheres is homotopy equivalent to~$\BSO(2)\simeq\BTT\simeq\CC P^\infty$, because this family can be modelled as
\[
\rmS^2=\SO(3)/\SO(2)\longrightarrow\BSO(2)\longrightarrow\BSO(3).
\] 

There is an equivalence~$\calG(\rmS^2)\simeq\TT$, and the extension
\[
\calG(\rmS^2)\longrightarrow\Aut(\rmS^2,d)\longrightarrow\Diff(\rmS^2)
\]
is equivalent to the split extension
\[
\TT\to\TT\times\SO(3)\to\SO(3),
\]
if the degree~$d$ is even, or to
\[
\TT\to\U(2)=\Spinc(3)\to\SO(3),
\]
if the degree~$d$ is odd, compare Remark~\ref{rem:gerbe}. 

Because~$\Aut(\rmS^2,d)$ is connected in both of the cases, the~(flat) Hodge bundle over the classifying space~$\BAut(\rmS^2,d)$ is trivial. Since the Picard torus~$\Pic(\rmS^2,d)$ is a point in this case, the Riemann-Roch bundle is trivial as well. Therefore, the vortex class for a family that is classified by a map~\hbox{$B\to\BAut(\rmS^2,d)$} lives in the group
\[
[\rmS^{(d+1)\underline{\CC}_B}_B,\rmS^1_B]^\TT_B,
\]
which is isomorphic to~$[\Sigma^{(d+1)\CC}B_+,\rmS^1]^\TT$, again by a base change adjunction. These groups may very well be non-trivial, even rationally.

For each fibre, the family invariant gives rise to an element in ~$[\rmS^{(d+1)\CC},\rmS^1]^\TT$ by restriction. In this way, we get the invariant of the fibre back, compare~\cite[Section~3.1]{Szymik:families}.

%%%%%%%%%%%%%%%%%%%%%%%%%%%%%%%%%%%%%%%%%%%%%%%%%%%%%%%%%%%%%%%%%%%%%%

%%%%%%%%%%%%%%%%%%%%%%%%%%%%%%%%%%%%%%%%%%%%%%%%%%%%%%%%%%%%%%%%%%%%%%

\vfill

\parbox{\linewidth}{Department of Mathematical Sciences\\
University of Copenhagen\\
Universitetsparken 5\\
2100 Copenhagen \O\\
DENMARK\\
\phantom{ }\\
\href{mailto:szymik@math.ku.dk}{szymik@math.ku.dk}}\\
\phantom{ }\\
\href{http://www.math.ku.dk/~xvd217}{www.math.ku.dk/$\sim$xvd217}

\end{document}